\documentclass[12pt]{amsart}
\pdfoutput=1

\usepackage{graphicx}
\usepackage{enumerate}
\usepackage{tikz}
\usepackage{wrapfig}
\usepackage{pinlabel}

\newtheorem{theorem}{Theorem}[section]

\newtheorem{corollary}[theorem]{Corollary}

\theoremstyle{definition}

\theoremstyle{remark}
\newtheorem{remark}[theorem]{Remark}

\numberwithin{equation}{section}

\newcommand{\storus}{S^1 \times D^2}

\newcommand{\G}{\mathcal{G}}
\newcommand{\A}{\mathcal{A}}

\newcommand{\h}[1]{\text{H}_1(#1)}

\newcommand{\even}{Y_{\A,0}}
\newcommand{\odd}{Y_{\A,1}}

\begin{document}

\title{On Krebes' tangle}
\author{Susan M. Abernathy}
\address{Mathematics Department\\
Louisiana State University\\
Baton Rouge, Louisiana}
\email{sabern1@tigers.lsu.edu}

\date{}


\begin{abstract}
A genus-$1$ tangle $\G$ is an arc properly embedded in a standardly embedded solid torus $S$ in the 3-sphere. We say that a genus-$1$ tangle embeds in a knot $K \subseteq S^3$ if the tangle can be completed by adding an arc exterior to the solid torus to form the knot $K$.  We call $K$ a closure of $\G$.  An obstruction to embedding a genus-1 tangle $\G$ in a knot is given by torsion in the homology of branched covers of $S$ branched over $\G$.  We examine a particular example $\A$ of a genus-1 tangle, given by Krebes, and consider its two double-branched covers.  Using this homological obstruction, we show that any closure of $\A$ obtained via an arc which passes through the hole of $S$ an odd number of times must have determinant divisible by three.  A resulting corollary is that if $\A$ embeds in the unknot, then the arc which completes $\A$ to the unknot must pass through the hole of $S$ an even number of times.
\end{abstract}



\maketitle

\section{Introduction}
\label{intro}
In this paper, we outline a method for computing the homology of the two double-branched covers of any properly embedded arc in the solid torus.  In particular, we use this method to partially answer a question posed by Krebes in \cite{kr}.

Let $S$ be a standardly embedded solid torus $\storus \subset S^3$.  Then a genus-$1$ tangle is a properly embedded arc in $S$.  Just as we may discuss embedding ordinary tangles in $B^3$ into knots and links (see \cite{kr}, \cite{psw}, and \cite{rub}), we may consider embedding genus-1 tangles in knots.  We say that a genus-1 tangle $\G$ embeds in a knot $K$ if $\G$ can be completed by an arc exterior to $S$ to form the knot $K$; that is, there exists some arc in $S^3 -Int(S)$ such that upon gluing this arc to $\G$ along their boundary points, we have a knot in $S^3$ which is isotopic to $K$.  We say that $K$ is a closure of $\G$.

Let $l$ denote a longitude for $S$ which is contained in $\partial S$ and avoids the genus-1 tangle.  A closure $K$ of $\G$ is called odd (respectively, even) with respect to $l$ if $lk(K,l)$ is odd (respectively, even).  If $l$ is chosen to be the longitude which circles the central hole of $S$ as in Fig. \ref{krebestangle}, and we span the longitude $l$ by a disk $\Delta$ filling the hole, then $lk(K,l)$ is the number of transverse intersections counted with sign of the arc which completes $\G$ to $K$ with $\Delta$. Thus, in this case we can say more colloquially that $K$ is an odd (respectively, even) closure with respect to $l$ if the arc which completes $\G$ to $K$ passes through the hole of $S$ an odd (respectively, even) number of times.

\begin{figure}
\labellist
\small\hair 2pt
\pinlabel $l$ at 377 185
\endlabellist
\begin{center} \includegraphics[height=1in]{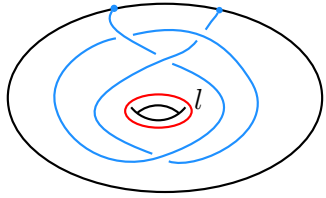}\end{center}
\caption{Krebes' genus-1 tangle $\mathcal{A}$ in $S$ together with a specified longitude $l$.}\label{krebestangle}
\end{figure}

In \cite{kr}, Krebes asks whether the genus-1 tangle $\A$ given in Fig. \ref{krebestangle} embeds in the unknot.  Using the following results from \cite{rub}, we are able to partially answer this question.  Note that when discussing this example, we always use the longitude $l$ drawn in Fig. \ref{krebestangle}. 

\begin{theorem}[Ruberman]\label{ruberman}
Suppose $M$ is an orientable 3-manifold with connected boundary, and $i:M \hookrightarrow N$ where $N$ is an orientable 3-manifold with $H_1(N)$ torsion. Then the inclusion map $i_\ast$ induces an injection of the torsion subgroup $T_1(M)$ of $\h{M}$ into $H_1(N)$.
\end{theorem}

This theorem has a useful corollary which can easily be proved directly using a Meyer-Vietoris sequence.

\begin{corollary}[Ruberman]\label{rubermancor}
Let $M$ and $N$ be as in Theorem \ref{ruberman} but suppose $H_1(N) = 0$.  Then $H_1(M)$ is torsion-free.
\end{corollary}

One obtains an obstruction to embedding genus-1 tangles in knots from Theorem \ref{ruberman} by applying the result to branched covers of $S$ branched over genus-1 tangles.

For any genus-1 tangle $\G$, the homology $\h{S-\G}$ is free abelian on the two generators given by the meridian $m$ of $\G$ and the longitude $l$ of $S$.  For a given $n$, each $n$-fold cover of $S$ branched over $\G$ is associated to a homomorphism $\varphi: \h{S-\G} \rightarrow \mathbb{Z}_n$ which maps $m$ to one. The remaining generator $l$ of $\h{S-\G}$ may be sent to any element of $\mathbb{Z}_n$; we use $\varphi (l)$ to index the $n$-fold branched covers.  So, $Y_{\G,i}$ denotes the $n$-fold cover of $S$ branched over $\G$ associated to the homomorphism $\varphi$ which maps $l$ to $i$.

If a genus-1 tangle $\G$ embeds in a knot $K$, then the $n$-fold cover $X_K$ of $S^3$ branched over $K$ restricts to some $n$-fold cover $Y_{\G,i}$ of $S$ branched over $\G$.  In this case, we say that the closure $K$ induces the cover $Y_{\G,i}$. Then according to Theorem \ref{ruberman}, the torsion subgroup $T_1(Y_{\G,i})$ of $\h{Y_{\G,i}}$ injects into $H_1(X_K)$.

Note that if $K$ is the unknot, then $X_K$ is $S^3$ and according to Corollary \ref{rubermancor}, the torsion subgroup $T_1(Y_{\G,i})$ is trivial. Thus, if there is any torsion in the homology of $Y_{\G,i}$, then any closure of $\G$ which induces the cover $Y_{\G,i}$ is not the unknot.

Ruberman considered using Theorem \ref{ruberman} to study Krebes' question about the genus-1 tangle $\A$. He found that ``it seems that the homology of all the cyclic covers of the solid torus, branched along this arc, is torsion-free" (see Section 5 of \cite{rub}).  However, our detailed computation in Sections 2 and 3 shows that the two-fold branched cover $Y_{\A,1}$ does have torsion in its homology, and we prove the following results.

\begin{theorem}\label{maintheorem}
If a knot $K$ in $S^3$ is an odd closure of $\A$, then $\det(K)$ is divisible by $3$.
\end{theorem}

\begin{corollary}\label{maincor}
If $\A$ embeds in the unknot, then the unknot is an even closure of $\A$.
\end{corollary}

\begin{figure}
\begin{center}
\includegraphics[height=1.3in]{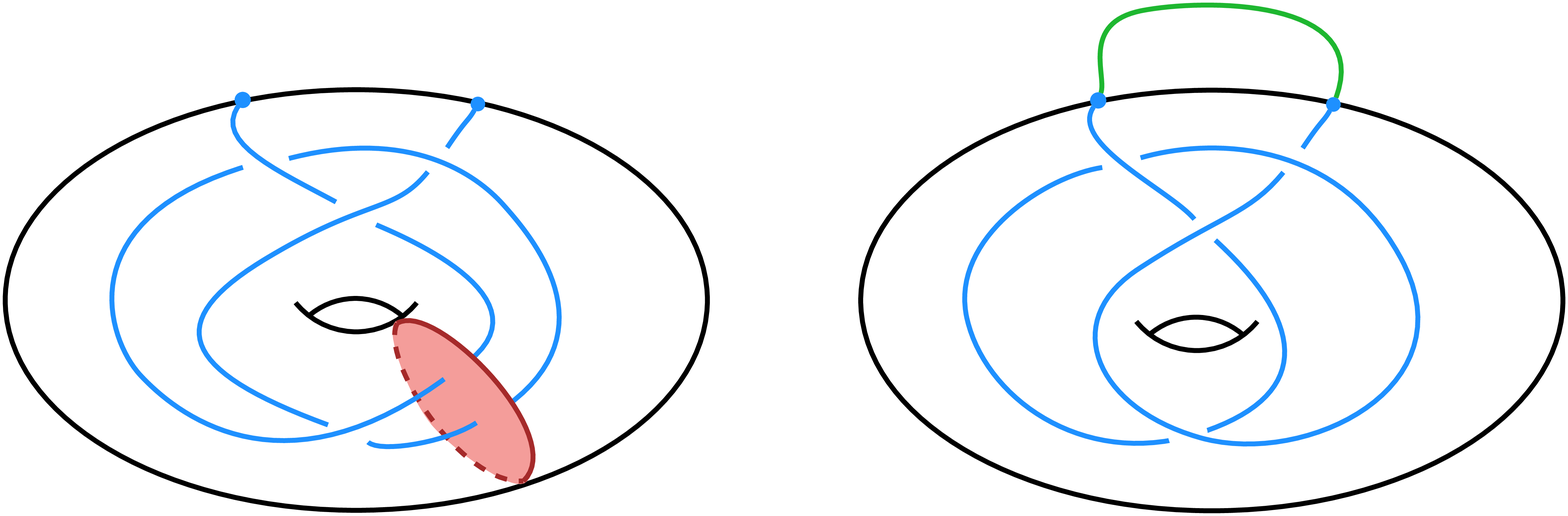}
\caption{The disk in $S$ where we perform a meridional twist, and the genus-1 tangle which results from the twist.}\label{krebestwist}
\end{center}
\end{figure}

Before further discussion, we need to make a remark about the definition of genus-1 tangles.
\begin{remark} \label{embeddingremark}
Note that when defining genus-1 tangles, we fix a standardly embedded solid torus $S$ in the 3-sphere.  The reason that we restrict to a fixed embedding is that there are many ways to re-embed a solid torus inside $S^3$.

For instance, if we perform a meridional twist on $S$ along the disk indicated in Fig. \ref{krebestwist}, the image of $\A$ under this twist can be easily seen to embed in an unknot via the exterior arc pictured in Fig. \ref{krebestwist}.  Thus it is necessary to specify the embedding of $\storus$ in the case of genus-1 tangles, and we restrict to a fixed standardly embedded solid torus in our definition.
\end{remark}


\section{Surgery descriptions for double-branched covers}\label{}

For the purposes of this paper, we restrict our attention to double-branched covers of $S$ branched over $\A$.  Since a homomorphism $\varphi: \h{S-\A} \rightarrow \mathbb{Z}_2$ must map the specified longitude $l$ to either zero or one, there are two double-branched covers, $\even$ and $\odd$.  We call $\even$ the even double-branched cover because it is induced by all even closures of $\A$ (with respect to $l$).  Similarly, since $\odd$ is induced by all odd closures of $\A$, we call it the odd double-branched cover.

\begin{figure}
 \includegraphics[height=1in]{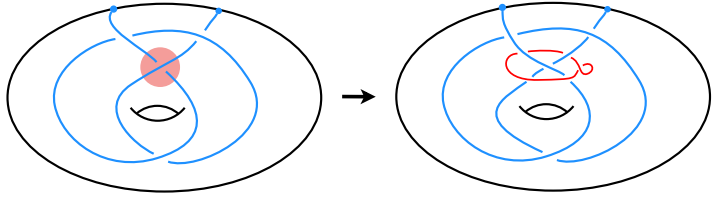}\caption{We perform surgery around a crossing, following \cite{rolf}.}\label{surgery}
\end{figure}

In this section, we adapt Rolfsen's technique to find surgery descriptions for these double-branched covers.

Following \cite{rolf}, we perform surgery near a carefully selected crossing (see Fig. \ref{surgery}) in such a way that after surgery we may essentially unwind $\A$ (via sliding its endpoints around the boundary in the complement of $l$) so that it looks trivial.  This process, illustrated in Fig. \ref{unwind}, results in a nice surgery description of $\A$ inside $S$.  Note that in the last drawing of Fig. \ref{unwind}, we choose to draw this surgery description in a particular way because it makes constructing branched covers easier.

\begin{figure}
 \includegraphics[width=4.85in]{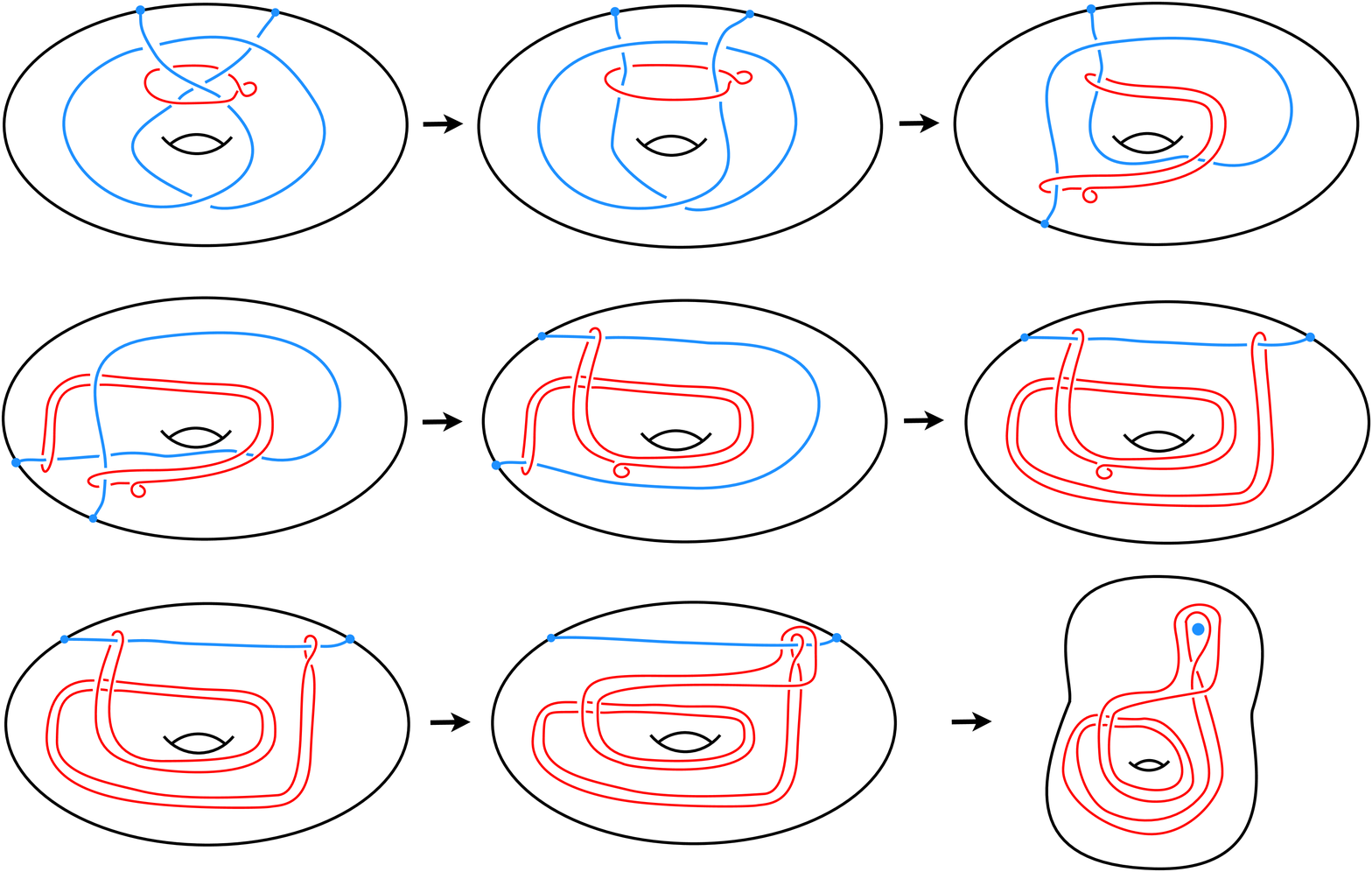}\caption{Unwinding the genus-1 tangle $\A$ to make it look trivial. The surgery curve is always given the blackboard framing.}\label{unwind}
\end{figure}

Now we construct the odd cover, $\odd$.  Construction is dictated by the homomorphism $\varphi: \h{S-\A} \rightarrow \mathbb{Z}_2$ corresponding to the cover.  If $\varphi$ maps a generator of $\h{S-\A}$ to a non-zero element, then we cut the solid torus along a disk transverse to that generator. Thus, we have two cuts to make in the case of the odd cover.

First, we cut $S$ along a disk which is transverse to the meridian $m$ of $\A$ and whose boundary is made up of the unwound genus-1 tangle $\A$ together with an arc in $\partial S$.  Then, because $\varphi$ sends $l$ to one, we cut $S$ along a disk which is transverse to $l$ and whose boundary is contained in $\partial S$. We then take two copies of the resulting manifold and glue them together carefully to obtain a surgery description for $\odd$.  This process is illustrated in Fig. \ref{constructodd}.

Although it is not needed in the proof of Theorem \ref{maintheorem}, we also give a surgery description of the even double-branched cover $\even$ in Fig. \ref{constructeven}.

\begin{figure}
\includegraphics[width=3.9in]{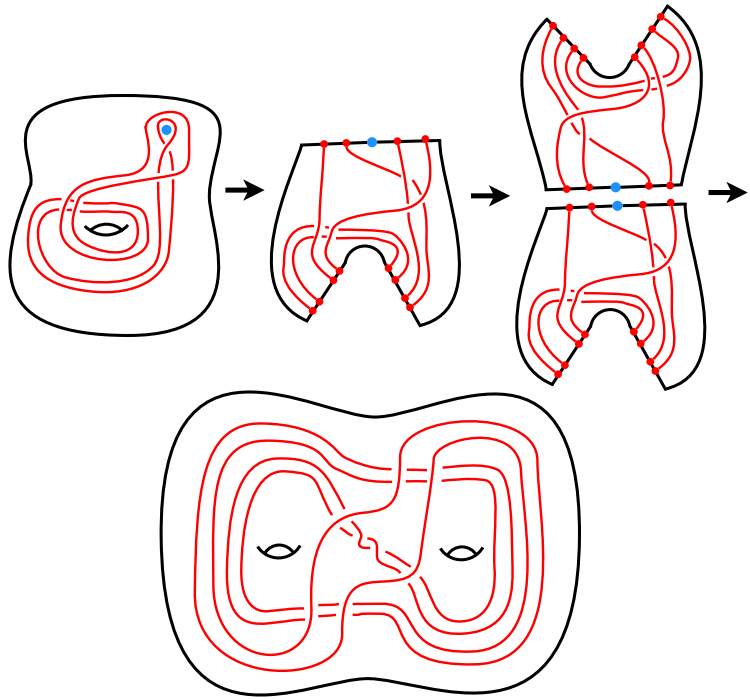}\caption{Constructing the odd double-branched cover $\odd$ of $S$ branched over $\A$.}\label{constructodd}
\end{figure}

\begin{figure}
\includegraphics[width=4.1in]{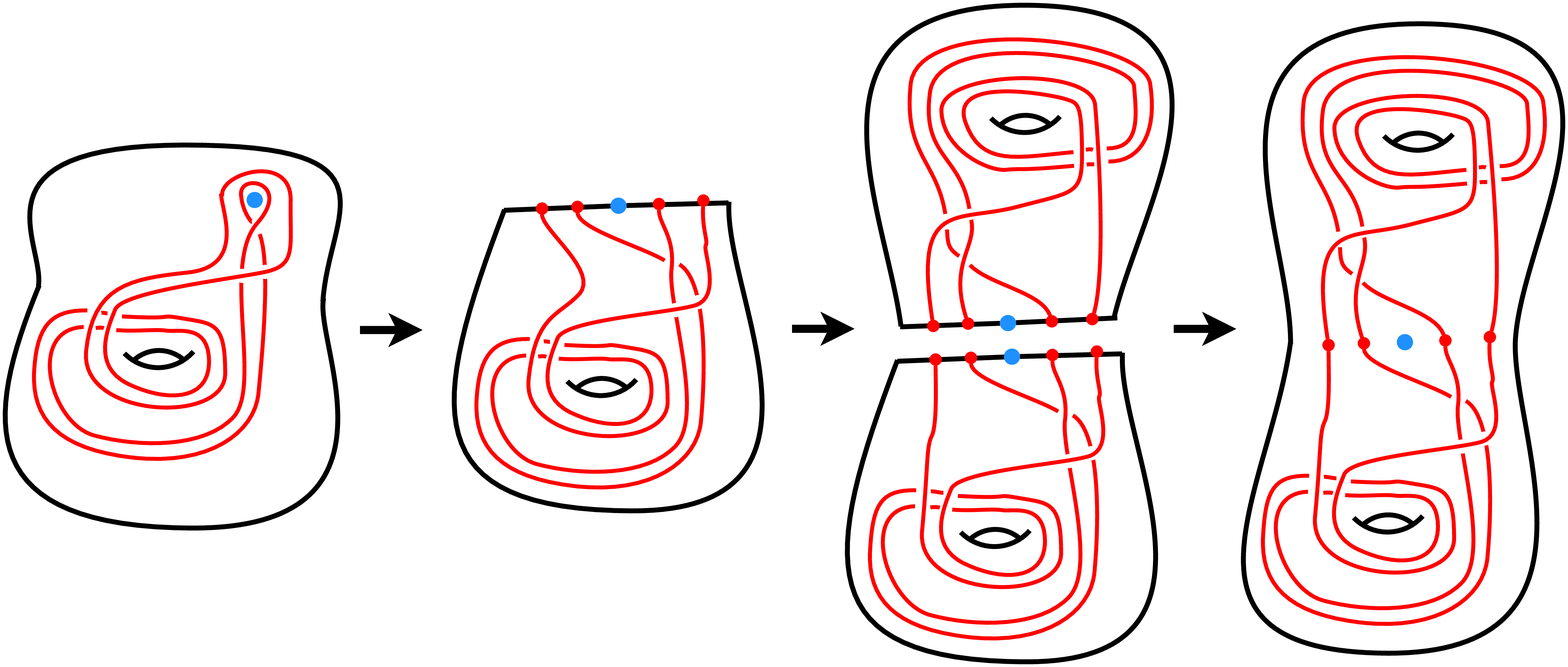}\caption{Obtaining a surgery description of $\even$.}\label{constructeven}
\end{figure}


\section{Homology of the covers}

Now we compute the homology of the odd double-branched cover.  From Fig. $\ref{constructodd}$ we see that the surgery description for $\odd$ is given by a 2-component surgery link inside a genus-2 handlebody.  We denote the components of the surgery link by $\sigma$ and $\tau$, and let $H$ denote the genus-2 handlebody.  The complement of $H$ in $S^3$ is a neighborhood of the handcuff graph $G$, pictured in Fig. \ref{oddhomology}, which is composed of loops $\alpha_1$ and $\alpha_2$ joined together by an arc.  Then the complement of $\sigma \cup \tau$ in $H$ can be viewed as the complement of $\sigma \cup \tau \cup G$ in $S^3$.  One can see that $\h{S^3 - (\sigma \cup \tau \cup G)}$ is isomorphic to $\h{S^3-(\sigma \cup \tau \cup \alpha_1 \cup \alpha_2)}$ which is free on four generators: the meridians of $\sigma$, $\tau$, $\alpha_1$, and $\alpha_2$.

\begin{figure}
\labellist
\small\hair 2pt
\pinlabel $\sigma$ at 108 414
\pinlabel $\tau$ at 527 424
\pinlabel $\alpha_1$ at 0 245
\pinlabel $\alpha_2$ [t] at 646 259
\pinlabel $G$ at 671 116
\endlabellist
\includegraphics[width=3in]{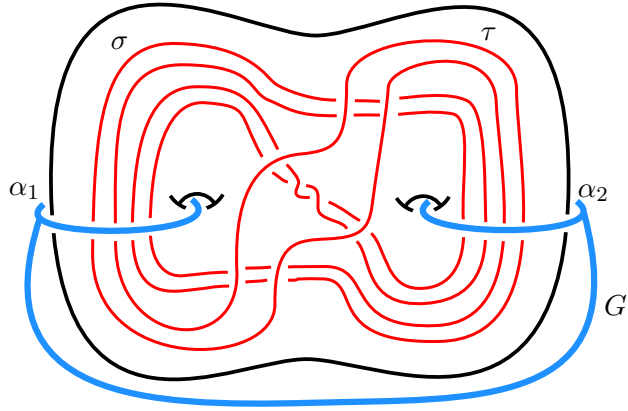}\caption{A surgery description of $Y_{\A,1}$.}\label{oddhomology}
\end{figure}

Completing the surgery by gluing in two solid tori according to $\sigma$ and $\tau$ introduces two relations on these four generators, which are given by the linking numbers of $\sigma$ and $\tau$ with each of $\sigma$, $\tau$, $\alpha_1$, and $\alpha_2$. Then $\h{\odd}$ is isomorphic to $\h{S^3-(\sigma \cup \tau \cup \alpha_1 \cup \alpha_2)}$ modulo these two relations, and we can get a presentation for $\h{\odd}$ using linking numbers.  Thus, we have the following presentation matrix for $\h{\odd}$:
$$\bordermatrix{\text{}&\sigma & \tau & \alpha_1 & \alpha_2 \cr
                \sigma & 1 &  2  & 0 & 0 \cr
                 \tau & 2  &  1 & 0 & 0 \cr}.$$

Using row and columns operations we obtain a simpler presentation matrix:
$$\bordermatrix{\text{}&\sigma & \tau & \alpha_1 & \alpha_2 \cr
                \sigma & 1 &  0  & 0 & 0 \cr
                 \tau & 0  &  3 & 0 & 0 \cr}.$$
Therefore, $\h{\odd}=\mathbb{Z} \oplus \mathbb{Z} \oplus \mathbb{Z}_3$ and we are now able to prove the main theorem.  Corollary \ref{maincor} follows immediately.

\begin{proof}[Proof of Theorem \ref{maintheorem}]
Let $K$ be an odd closure of $\A$, and let $X_K$ denote the double cover of $S^3$ branched over $K$.  Since $K$ is an odd closure of $\A$, it induces a restriction from $X_K$ to $\odd$.  Then according to Theorem \ref{ruberman}, we have that $T_1(\odd) = \mathbb{Z}_3 \hookrightarrow \h{X_K}$.  Thus $|T_1(\odd)|=3$ divides $|\h{X_K}| = \det(K)$.
\end{proof}

We are unable to use this method to restrict all closures of $\A$ because $\even$ has a torsion-free first homology group.  Indeed, the statement in Remark \ref{embeddingremark} allows us to see that the even cover does embed in $S^3$ and so must have torsion-free first homology.  Of course, this can be verified by deriving a presentation for the homology of $\even$ using the procedure above.

\end{document}